\documentclass{siamltex}
\usepackage{amssymb}
\usepackage{amsmath}
\usepackage{amsfonts}
\usepackage{graphicx}
\usepackage{xcolor}
\usepackage{float}
\usepackage{multicol}
\usepackage{subcaption}
\usepackage{multirow}
\title{Energy dissipation rates
of ensemble eddy viscosity models of turbulence: the periodic box\thanks{Submitted to the editors DATE.{The research herein was partially supported by NSF grant DMS 241089.}}}
\author{William Layton\thanks{ 
Department of Mathematics, University of Pittsburgh, Pennsylvania, 15260, USA, (wjl@pitt.edu).}
\and Nanda Nechingal Raghunathan\thanks{ 
Department of Mathematics, University of Pittsburgh, Pennsylvania, 15260, USA  ({nan158@pitt.edu}).}}

\date{}

\begin{document}
\maketitle
\begin{abstract}
Classical eddy viscosity models of turbulence add an eddy viscosity term
based on the Kolmogorov-Prandtl parameterization by a turbulent length scale 
$l$ and a turbulent kinetic energy $k^{\prime }$. Approximations of the
unknowns $l,k^{\prime }$ are typically constructed by solving
multi-parameter systems of nonlinear convection-diffusion-reaction
equations. Often these over-diffuse so additional fixes are added.
Alternately, one can solve an ensemble of NSE's with perturbed data and
simply compute directly $k^{\prime }$(without modeling). The question then
arises: Does this ensemble eddy viscosity approach over-diffuse solutions?
We prove herein that for turbulence in a periodic box it does not.
\end{abstract}


\section{ Introduction}\label{sec1}

\begin{center}
      \textit{We dedicate this paper to Max Gunzburger.\\ 
      He inspired the adventure we continue here, \\
      and was a mentor and friend along the way!}
\end{center}

In the numerical simulation of flows incomplete data, quantification of
uncertainty, Cheung, Oliver, Prudencio, Prudhomme, Moser \cite{COPPM11},
limits on forecasting skill, Kalnay \cite{K03}, quantification of flow
sensitivities, Martin, Xue \cite{MX06}, error estimation, Fortin, Abaza,
Anctil, Turcotte \cite{FAT14}\,  and other issues, Leutbecher and Palmer 
\cite{LP10}, lead to the problem of computing ensembles of velocities and
pressures. At higher Reynolds numbers, each realization is approximated by
adding to a discretization a turbulence model with eddy viscosity $\nu_{\text{turb}}\mathbf{(\cdot )}$. \ Thus, a continuum turbulence model (like (\ref%
{eq:EEVmodel}) below) exists as an intermediate idea between a physically
realized turbulent flow and it's numerical simulation. Nevertheless, the
analysis of a continuum model (herein) is valuable for delineating 
both
positive and negative aspects of the numerical solution independent of
factors such as grid orientation, solvers, stopping criteria and so on.

The velocity and pressure ensemble $u_{j}=u(x,t;\omega
_{j}),p_{j}=p(x,t;\omega _{j}),j=1,\cdot \cdot \cdot ,J,$ in the flow domain 
$\Omega $\ satisfy\footnote{%
To simplify a non-essential feature of the analysis associated with Korn's
inequality we analyze the EV model with the full gradient rather than the
deformation tensor.}
\begin{align}
    \label{eq:EEVmodel}
    \begin{cases}
        u_{t}+u\cdot \nabla u-\nu \triangle u-\nabla \cdot \left( \nu_{\text{turb}}\mathbf{%
        (\cdot )}\text{ }\nabla u\right) +\nabla p
        =f(x;\omega _{j}), \\ 
        \nabla \cdot u=0, \text{ and }\\
        u(x,0;\omega _{j})=u^{0}(x;\omega _{j})\text{}%
    \end{cases}
\end{align}
plus boundary conditions on $\partial \Omega $. Here $\nu $ is the kinematic viscosity, $\omega _{j}$\ are the sampled values determining the ensemble data and $\nu _{\text{turb}}$ is the turbulent viscosity parameter. The modeling problem becomes one of determining the scalar $\nu_{\text{turb}}$ in terms of the flow variables. Ensemble data can be used for model parameterization, as proposed already in 2002 by Carati, Rogers and Wray \cite{CRW02} and developed starting with \cite{JL14}, \cite{JLM20}.\\

The ensemble mean\textbf{\ }$\left\langle u\right\rangle _{e}$\textbf{,}
fluctuation\textbf{\ }$u_{j}^{\prime }$, its magnitude $|u^{\prime }|_{e}$,
induced turbulent kinetic energy (TKE) density $k^{\prime }$\ and, turbulence length scale are%
\begin{align*}
    \text{\textbf{ensemble average} }u 
    &:  
    \left\langle u\right\rangle_{e}(x,t)
     = \frac{1}{J}\sum_{j=1}^{J}u(x,t;\omega _{j}), \\ 
    \text{\textbf{fluctuation}} 
    &: u^{\prime }(x,t;\omega _{j})
     =  u(x,t;\omega_{j})-\left\langle u\right\rangle _{e}(x,t) \\ 
    \text{\textbf{fluctuation magnitude}}
    &:  |u^{\prime }|_{e}^{2}(x,t)
    =  \frac{1}{J}\sum_{j=1}^{J}|u^{\prime }(x, t;\omega_j)|^{2}, \\ \text{\textbf{turbulent kinetic energy}}
    &:  k^{\prime }(x,t)=  \frac{1}{2}\text{ }|u^{\prime }|_{e}^{2}(x,t) \\ 
    \text{\textbf{turbulence length scale}} 
    &: l(x,t) \quad \text{to be specified}%
\end{align*}

Fluctuations' effects on the mean flow were envisioned by Boussinesq \cite%
{B77} as a mixing or dissipative process; see \cite{L14}, \cite{LL16}, \cite%
{JLM20} for proofs of the correctness of his vision. Consistently, $\nu_{\text{turb}}(\cdot )$ should depend on a dimensionally consistent combination of
a local length scale $l$ and the turbulent kinetic energy density $k^{\prime
}$ and should increases with $k^{\prime }$. The resulting Kolmogorov-Prandtl
relation, e.g. Pope \cite{Pope}, gives for $\nu_{\text{turb}}\mathbf{(\cdot ),}$%
\footnote{%
With $l$ the wall normal distance, Pope \cite{Pope} deduces $\mu =0.55$ from
the law of the wall. With $l=|u^{\prime }|\tau $ both this derivation and
the classic idea of Lilly \cite{L67} do not directly apply to determine $\mu $.}%
\begin{equation*}
\nu_{\text{turb}}\mathbf{(\cdot )}\text{ }=\mu l\sqrt{k^{\prime }}\text{.}
\end{equation*}%
With an ensemble eddy viscosity (EEV) model, the TKE density $k^{\prime
}(x,t)$\ is directly calculated (not modelled as in RANS and URANS
approaches, e.g., Wilcox \cite{Wilcox}).

The \textit{turbulence length scale} $l(x,t)$\ plays a role similar to the
mean free pass in kinetic theory. It represents a distance fluctuations must
travel to interact with each other or similarly the distance fluctuations
travel in one (small) time unit $\tau $. The motivating approach to the
turbulence length scale herein (the distance a fluctuation travels in one
time unit) was developed by Kolmogorov \cite{K42} (see also Spalding \cite%
{S91} for an interesting historical perspective), mentioned by Prandtl \cite%
{P26} and is still of practical use, e.g., Teixeira and Cheinet \cite{T01}.
It is motivated by the use of a turbulence model for under-resolved
simulations with a small time scale $\tau $ (related to the time step). This
choice leads to the local length scale and eddy viscosity 
\begin{equation}
l(x,t)=|u^{\prime }|_{e}\tau \quad \text{ yielding } \quad \nu_{\text{turb}}\mathbf{(}x,t%
\mathbf{)}\text{ }=\mu |u^{\prime }|_{e}^{2}(x,t)\tau .
\label{EQ:lengthScale}
\end{equation}%
As $|u^{\prime }|_{e}^{2}$\ is independent of $\omega _{j}$\ the turbulent
viscosity\ will be the same for all realizations (i.e., independent of $%
\omega _{j}$). This reduces the computational cost of solving for an
ensemble of model solutions with ensemble algorithms developed starting with 
\cite{JL13}.

EV models  like \eqref{eq:EEVmodel} have two main failure modes. The\textit{\ lesser} one (not
analyzed herein) is that eddy viscosity cannot account for intermittent
energy flow from fluctuations back to the mean velocity. The \textit{primary}
failure mode is \textit{over-dissipation} of solutions leading to a lower
Reynolds number flow, e.g., Sagaut \cite{Sagaut}, Wilcox \cite{Wilcox}.
Since the TKE is directly calculated, it is hoped EEV, (\ref{eq:EEVmodel})
above, has energy dissipation rate comparable (uniformly in the Reynolds
number) to the energy input rate, $\mathcal{O}(U^{3}/L)$. To analyze model
dissipation away from walls we consider (\ref{eq:EEVmodel}) subject to
periodic boundary conditions: on $\Omega =(0,L_{\Omega })^{3}$, for $\,\phi
=u,\,$ $u^{0},\,f,\,p,$ 
\begin{equation}
\phi (x+L_{\Omega }e_{j},t)=\phi (x,t)\quad j=1,2,3\text{ and}\int_{\Omega
}\phi\, dx=0\,.  \label{eq:PeriodicBCs}
\end{equation}%

Section \ref{sec_notaion_and_prelim} defines needed parameters and gives \'{a} priori bounds on model
solutions. Section \ref{sec_model_existence} (briefly) gives an of existence proof for model
solutions. The energy input rate, (\ref{EQenergyInput}) of Section \ref{Energy Dissipation and Turbulence Phenomenology}
below, is $\mathcal{O}(U^{3}/L)$. The main result herein, Theorem \ref{main_theorem}, is that the EEV model (\ref{eq:EEVmodel}) does not
over-dissipate. Specifically, we prove that,
uniformly in the Reynolds number,%
\begin{equation*}
\text{time average}\left( \frac{1}{|\Omega |}\int_{\Omega }[\nu +\nu_{\text{turb}}]|\nabla u|^{2}\,dx\right) \leq C\frac{U^{3}}{L}.
\end{equation*}

The main assumption of Theorem \ref{main_theorem} is that model solutions satisfy
an energy inequality. To explain this assumption, taking the $L^{2}(\Omega )$
inner product of (\ref{eq:EEVmodel}) with $u$ shows that a sufficiently
regular realization satisfies the energy equality%

\begin{align}
\label{energy_equality}
    \frac{1}{2}\int_{\Omega }|u(x,
    & T;\omega _{j})|^{2}\,
     dx
    +\int_{0}   ^{T}\int_{\Omega }\nu |\nabla u(x,t;\omega _{j})|^{2}+\mu       \tau
    |u^{\prime }(x,t)|_{e}^{2}|\nabla u(x,t;\omega _{j})|^{2}\,dx\,dt \notag \\
     &=\frac{1}{2}\int_{\Omega }|u^{0}(x;\omega
    _{j})|^{2}\,dx+\int_{0}^{T}\int_{\Omega }f(x,t;\omega _{j})\cdot u(x,t;\omega
    _{j})\,dx\,dt.
\end{align}

The existence theory for this model contains many open problems, Section \ref{sec_model_existence}.
Herein, we shall assume that data $u^{0},f\in L^{2}(\Omega )$ is smooth and divergence free. We treat weak solutions that satisfy the energy \textit{in}equality
consistent with \eqref{energy_equality}.
\begin{align}
\label{EQEnergyIneq}
     \frac{1}{2}\int_{\Omega }|u(x,
     &T;\omega _{j})|^{2}\,dx
    +\int_{0}^{T}\int_{\Omega }\nu |\nabla u(x,t;\omega _{j})|^{2}+\mu \tau
    |u^{\prime }(x,t)|_{e}^{2}|\nabla u(x,t;\omega _{j})|^{2}\,dx\,dt \notag \\ 
    &\leq \frac{1}{2}\int_{\Omega }|u^{0}(x;\omega_{j})|^{2}\,dx+\int_{0}^{T}\int_{\Omega }f(x,t;\omega _{j})\cdot u(x,t;\omega_{j})\,dx\,dt, 
\end{align}
and thus, 
\begin{align}
\label{EQEnergynew}
 &\frac{1}{2} \left\langle \frac{1}{|\Omega |} 
 \int_{\Omega }|  u(x,T;\omega_{j})|^{2}\,dx\right\rangle _{e} \\
 &+\int_{0}^{T}\left\langle \frac{1}{|\Omega |}\int_{\Omega }
 \left(\nu 
 +\mu \tau |u^{\prime }(x,t)|_{e}^{2}|\right)
 \nabla u(x,t;\omega _{j})|^{2}\,dx\right\rangle _{e}dt \notag\\ 
&\leq \frac{1}{2}\left\langle \frac{1}{|\Omega |} \int_{\Omega
}|u^{0}(x;\omega _{j})|^{2}\,dx\right\rangle _{e}+\int_{0}^{T}\left\langle 
\frac{1}{|\Omega |}\int_{\Omega }f(x,t;\omega _{j})\cdot u(x,t;\omega
_{j})\,dx\,\right\rangle _{e}dt \notag.
\end{align}

\subsection{Energy Dissipation and Turbulence Phenomenology\label{Energy Dissipation and Turbulence Phenomenology}}

The energy dissipation rate is critical in lift, drag and the Kolmogorov
theory of turbulence. We briefly review (from e.g. Pope \cite{Pope},
Davidson \cite{D15}, see Lewandowski \cite{L16} for what can be proven
directly from the Navier-Stokes equations) this last connection. It is known
from flow data that energy is concentrated in large scales and energy
dissipation occurs primarily at small scales. Consistently, the K41 theory
posits that a smooth body force inputs energy into a flow's large scales and
non-negligible energy dissipation occurs only at small scales. Let $U,L$
denote a large scale velocity magnitude and a large length scale. Kinetic
energy then scales like $U^{2}.$ The large scale turn over time $T^{\ast }$
is determined by $L=UT^{\ast }$. Thus, the rate (\textit{per unit time}) of
energy input at the large scales is 
\begin{equation}
\text{energy input rate} \simeq \frac{U^{2}}{T^{\ast }}=\frac{U^{3}}{L}%
.  \label{EQenergyInput}
\end{equation}%
Let, as usual, $\eta $\ denote the Kolmogorov microscale and $v_{\text{small}}$ the
associated velocity scale. Form two Reynolds numbers for the large and small
scales%
\begin{equation*}
\mathcal{R}e=\frac{LU}{\nu }\text{ and }\mathcal{R}e_{\text{small}}=\frac{\eta
v_{\text{small}}}{\nu }.
\end{equation*}%
The Kolmogorov microscale is determined by two postulates:

\begin{equation*}
\begin{array}{cc}
{\left\vert \frac{\text{\normalsize viscous term}}{\text{\normalsize nonlinear term}}\right\vert} _{%
\text{small scales}}\text{ }\simeq {\Large 1} & \frac{{\Large \eta v}_{%
{\Large \text{small}}}}{{\Large \nu }}{\Large \simeq 1} \\ 
\left. 
\begin{array}{c}
\text{Energy input rate} \\ 
\text{ at large scales}%
\end{array}%
\right. \simeq \text{ }\left. 
\begin{array}{c}
\text{Energy dissipation } \\ 
\text{rate at small scales}%
\end{array}%
\right.  & 
\begin{array}{c}
\frac{{\Large U}^{{\Large 3}}}{{\Large L}}\simeq \nu \left( \frac{{\Large v}%
_{{\Large \text{small}}}}{{\Large \eta }}\right) ^{2}\text{ } \\ 
\text{as }\nu |\nabla v_{\text{small}}|^{2}\simeq \nu \left( \frac{{\Large v}_{%
{\Large \text{small}}}}{{\Large \eta }}\right) ^{2}%
\end{array}%
\end{array}%
\end{equation*}%
These two postulates yield two equations for the two unknowns $\eta
,v_{\text{small}}$. Solving yields Kolmogorov's estimate of the smallest persistent
structure on a turbulent flow%
\begin{equation*}
\eta \simeq \mathcal{R}e^{-3/4}L.
\end{equation*}%
Matching energy dissipation rates (primarily at small scales) to the $U^{3}/L
$\ energy input rate (primarily at large scales) is a fundamental organizing
principle for turbulent flow statistics.

\subsection{Related work}

We cannot over-stress the importance of the NSE work of Constantin, Doering
and Foias \cite{CD92}, \cite{DF02} to the analysis herein. Their work has
been developed in many important directions (such as for shear flows in
general domains, Wang \cite{Wang97} and the analysis in Chow and Pakzad \cite%
{CP20} of when the upper bound is attained) for the NSE in subsequent years.
For \textit{turbulence models, upper bounds for energy dissipation rates\
have particular importance.} They yield model conditions that exclude the
over-dissipation failure mode. This application of the
Constantin-Doering-Foias theory has yielded useful results already for
several turbulence models, e.g., \cite{KLS22}, \cite{LM18}, \cite{L02}, \cite%
{LS20}, \cite{La16}.

For eddy viscosity models, Escudier \cite{E66} proposed in 1966 capping the
turbulent viscosity $\nu_{\text{turb}}\mathbf{(\cdot )}$, an idea that re-emerges
naturally in Borggaard, Iliescu and Roop \cite{BIR09}. It was proven in \cite%
{KLS21}\ that Escudier's uniform upper bound on $\nu_{\text{turb}}\mathbf{(\cdot )}
$ suffices to control model diffusion quite generally. In the EEV model
above $\nu_{\text{turb}}\mathbf{(}x,t\mathbf{)}$ can be unbounded for weak
solutions. Yet analysis herein still can prove a useful bound on model
energy dissipation rates. The existence theory in Section \ref{sec_model_existence} requires only
the bound $l\leq diameter(\Omega )$\ and not boundedness of $\nu _{\text{turb}}$.

\section{Notation and preliminaries}
\label{sec_notaion_and_prelim}

Our notation is standard and follows, e.g., Doering and Gibbon \cite{DG95}.
The flow domain is the open box $\Omega =(0,L_{\Omega })^{3}$ in $\mathbb{R}%
^{3}$. The $L^{2}(\Omega )$ norm and the inner product are $\Vert \cdot
\Vert $ and $(\cdot ,\cdot )$. Likewise, the $L^{p}(\Omega )$ norms and the
Sobolev $W_{p}^{k}(\Omega )$ norms are $\Vert \cdot \Vert _{L^{p}}$ and $%
\Vert \cdot \Vert _{W_{p}^{k}}$ respectively. $H^{k}(\Omega )$ is the
Sobolev space $W_{2}^{k}(\Omega )$, with norm $\Vert \cdot \Vert _{H^{k}}$. $%
C$ represents a generic positive constant independent of $\nu ,U,L$ and
other model parameters. Its value may vary from situation to situation.

Three kinds of averaging will be used, ensemble $\left\langle \cdot
\right\rangle _{e}$ , finite time $\left\langle \cdot \right\rangle _{T}$
and long time $\left\langle \cdot \right\rangle _{\infty }$. Ensemble
averaging, introduced above, is $\left\langle \phi \right\rangle _{e}:=\frac{%
1}{J}\sum_{j=1}^{J}\phi (\omega _{j})$. The two time averages are%
\begin{equation*}
\left\langle \phi \right\rangle _{T}=\frac{1}{T}\int_{0}^{T}\phi (t)dt\text{
and\ }\left\langle \phi \right\rangle _{\infty }=\limsup_{T\rightarrow
\infty }\left\langle \phi \right\rangle _{T}.
\end{equation*}%

These satisfy%
\begin{equation*}
\left\langle \phi \psi \right\rangle \leq \left\langle |\phi
|^{2}\right\rangle ^{1/2}\left\langle |\psi |^{2}\right\rangle
^{1/2},\left\langle \phi ^{\prime }\right\rangle _{e}=0,\text{ and }%
\left\langle \left\langle \phi \right\rangle _{e}\right\rangle
_{T}=\left\langle \left\langle \phi \right\rangle _{T}\right\rangle _{e}.
\end{equation*}%
We recall that uniform in $T$ bounds on the following quantities also follow
from the energy inequality (\ref{EQEnergyIneq}) and standard differential
inequalities.

\begin{proposition}[Uniform Bounds]
Consider the model (\ref{eq:EEVmodel}) with boundary conditions (\ref%
{eq:PeriodicBCs}) and $l(x,t)=|u^{\prime }|_{e}\tau $. For a weak solution
satisfying (\ref{EQEnergyIneq})\ the following are uniformly bounded in $T$
\begin{gather}
\label{EQ:bounds}
||u(T)||^{2}, \quad
\int_{\Omega }\,\nu _{\text{turb}}(\cdot ,T)\,dx, \quad
\frac{1}{T} \int_{0}^{T}\left( \int_{\Omega }|\,\nabla {u}|^{2}\,dx\right) \,dt,
 \quad||u^{\prime }(T)||^{2},   \notag\\
\quad\text{and } \quad\frac{1}{T}\int_{0}^{T}\left( \int_{\Omega }[\nu +\,\nu
_{\text{turb}}]|\,\nabla {u}|^{2}\,dx\right) \,dt.  \notag
\end{gather}
\end{proposition}
To develop the results for energy dissipation rates some scaling constants are needed. The setting
considered is that for smooth initial condition and body force. Over long
enough time turbulence will develop so the standard scaling parameters are
defined through infinite time limits as in the Constantin-Doering-Foias
theory \cite{CD92}, \cite{DF02}. Since energy is input at the large scales
by the smooth body force $f(x;\omega _{j})$, the large length scale $L$, \eqref{eq:ULscales} below, must involve both the domain size ($L_{\Omega }$) and the length scales where $%
f(\cdot )$ inputs energy.

\begin{definition}
\label{definition of scales}
The scales of the body force $F$, velocity $U$, fluctuation scale $U^{\prime
}$, length scale $L$ are then%
\begin{align*} 
 F
    &=\left\langle \frac{1}{|\Omega |}||f||^{2}\right\rangle _{e}^{\frac{1}{2}},\\ 
    U_T 
    & = \left\langle \left\langle \frac{1}{|\Omega |}%
    ||u||^{2}\right\rangle _{e}\right\rangle _{T}^{\frac{1}{2}},\\  
    U
    &=\left\langle \left\langle \frac{1}{|\Omega |}||u||^{2}\right\rangle_{e}\right\rangle _{\infty }^{\frac{1}{2}}, \\ 
    U_T^{\prime }
        &=\left\langle \left\langle \frac{1}{|\Omega |}||u^{\prime}||^{2}\right\rangle _{e}\right\rangle _{T }^{\frac{1}{2}} \\
    U^{\prime }
        &=\left\langle \left\langle \frac{1}{|\Omega |}||u^{\prime}||^{2}\right\rangle _{e}\right\rangle _{\infty }^{\frac{1}{2}} \\
    \label{eq:ULscales}
    L
    &=\min \left( L_{\Omega },\frac{F}{\max_{j}
    ||\nabla f(\cdot ;\omega_{j})||_{L^{\infty }}},\frac{F}{\left\langle \frac{1}{|\Omega |}
    ||\nabla f||^{2}\right\rangle _{e}^{\frac{1}{2}}}\right) 
\end{align*}%
The large scale turnover time $T^{\ast }$, turbulence intensity $I(u)$\ and
Reynolds number $\mathcal{R}e$ are%
\begin{equation*}
T^{\ast }=\frac{L}{U}\text{ , \ }I(u)=\left( \frac{U^{\prime }}{U}\right)
^{2}\text{\ and }\mathcal{R}e=\frac{LU}{\nu }.
\end{equation*}
\end{definition}

The uniform bounds in Proposition (\ref{EQ:bounds}) immediately imply that
the quantities in Definition \ref{definition of scales} (defined through limit superiors) are well
defined and finite. The turbulent intensity satisfies $I(u)\leq 1$ since $%
U^{\prime }\leq U$.
The large length scale $L$ has units of length and
satisfies%
\begin{equation}
\max_{j}||\nabla f||_{L^{\infty }}\leq \frac{F}{L}\text{ and }\left\langle 
\frac{1}{|\Omega |}||\nabla f||^{2}\right\rangle _{e}\leq \frac{F^{2}}{L^{2}}%
\text{ \ }.  \label{eq:FandLproperties}
\end{equation}

\subsection{The energy dissipation rate}

The problem data $u^{0}(x;\omega _{j}),f(x;\omega _{j})$\ are assumed
smooth, periodic, mean-zero and satisfy $\nabla \cdot u^{0}=0$ and $\nabla
\cdot f=0$. The model's ensemble averaged energy dissipation rate, from the
energy inequality (\ref{EQEnergyIneq}), is%
\begin{equation*}
\varepsilon (t):=\left\langle \frac{1}{|\Omega |}\int_{\Omega }\nu |\nabla
u(x,t;\omega _{j})|^{2}+\mu \tau |u^{\prime }(x,t)|_{e}^{2}|\nabla
u(x,t;\omega _{j})|^{2}\,dx\right\rangle _{e}
\end{equation*}%
It will be convenient to decompose the energy dissipation rate by $%
\varepsilon =\varepsilon _{\text{viscous}}+\varepsilon _{\text{turb}}$ where 
\begin{align*}
    \varepsilon _{\text{viscous}}
    &=\left\langle \frac{1}{|\Omega |}\int_{\Omega }\nu
    |\nabla u(x,t;\omega _{j})|^{2}\,dx\right\rangle _{e}, \\ 
    \varepsilon _{\text{turb}}
    &=\left\langle \frac{1}{|\Omega |}\int_{\Omega }\mu \tau
    |u^{\prime }(x,t;\omega _{j})|_{e}^{2}|\nabla u(x,t;\omega
_{j})|^{2}\,dx\right\rangle _{e}.%
\end{align*}%
We also define the terms $\left\langle \varepsilon \right\rangle _{T} $ and $\left\langle \varepsilon \right\rangle _{\infty}$ as follows: 
\begin{align*}
       \left\langle \varepsilon \right\rangle _{T} 
       &:= \frac{1}{T}\int_{0}^{T}\varepsilon _{\text{viscous}}+\varepsilon _{\text{turb}} \, dt \\
       \left\langle \varepsilon \right\rangle _{\infty}
       & : =\limsup_{T\rightarrow\infty} \left\langle \varepsilon \right\rangle _{T} 
\end{align*}
Since $\nu _{\text{turb}}$\ is independent of $\omega _{j}$, $\varepsilon _{\text{turb}}$
can be split into a term responding to the mean velocity and one solely
depending on the average fluctuation.
\begin{proposition}
We have%
\begin{equation*}
\varepsilon _{\text{turb}}=\frac{1}{|\Omega |}\int_{\Omega }\mu \tau |u^{\prime
}|_{e}^{2}|\nabla \left\langle u\right\rangle _{e}|^{2}\,dx+\frac{1}{|\Omega |}%
\int_{\Omega }\mu \tau |u^{\prime }|_{e}^{2}|\nabla u^{\prime }|_{e}^{2}\,dx.
\end{equation*}
\end{proposition}

\begin{proof}
This follows as $\left\langle |\nabla u(x,t;\omega _{j})|^{2}\right\rangle
_{e}=|\nabla \left\langle u(x,t;\omega _{j})\right\rangle _{e}|^{2}+|\nabla
u^{\prime }(x,t;\omega _{j})|_{e}^{2}$.
\end{proof}
We also define the terms $\varepsilon_T$ and $\varepsilon_\infty$ as follows: 
\begin{align*}
       \left\langle \varepsilon \right\rangle _{T} 
       &:= \frac{1}{T}\int_{0}^{T}\varepsilon _{\text{viscous}}+\varepsilon _{\text{turb}} \, dt \\
       \left\langle \varepsilon \right\rangle _{\infty}
       & : =\limsup_{T\rightarrow\infty} \left\langle \varepsilon \right\rangle _{T} 
\end{align*}

\section{Estimation of EEV energy dissipation rates}

The estimate below in the main theorem, $\left\langle \varepsilon
\right\rangle _{\infty }\lesssim $\ $U^{3}/L,$ is consistent as $%
\mathcal{R}e\rightarrow \infty $ and as $\tau \rightarrow 0$\ with both phenomenology,
surveyed in Section 1.1, see Pope \cite{Pope} for elaboration, and the rate
proven for the Navier-Stokes equations in \cite{CD92}, \cite{DF02}.

\begin{theorem}
\label{main_theorem}
Suppose $u^{0}(x;\omega _{j}),f(x;\omega _{j})$\ are smooth, periodic,
mean-zero $L^{2}(\Omega )$ functions satisfying satisfy $\nabla \cdot u^{0}=$
$\nabla \cdot f=0$. Let $u(x,t;\omega _{j})$ be a weak solution of (\ref%
{eq:EEVmodel}) satisfying the energy inequality. The time and ensemble
averaged rate of energy dissipation satisfies the following. For any $%
0<\alpha <1,$%
\begin{equation*}
\left\langle \varepsilon \right\rangle _{\infty }\leq \left( \frac{1}{%
1-\alpha }+\frac{1}{4\alpha (1-\alpha )}\mathcal{R}e^{-1}+\frac{\mu }{%
4\alpha (1-\alpha )}\frac{\tau }{T^{\ast }}I(u)\right) \frac{U^{3}}{L}.
\end{equation*}%
Thus, for $\alpha =1/2,\mu =0.55$%
\begin{equation*}
\left\langle \varepsilon \right\rangle _{\infty }\leq \left( 2+\mathcal{R}%
e^{-1}+0.55\frac{\tau }{T^{\ast }}I(u)\right) \frac{U^{3}}{L}.
\end{equation*}
\end{theorem}

\subsection{Proof of the theorem}

Dividing (\ref{EQEnergynew}) by $T$ gives%
\begin{align}
\label{EQ:first}
    \frac{1}{T}\left\langle \frac{1}{2}\frac{1}{|\Omega |}||u(T)||^{2} \right\rangle _{e}
    +\frac{1}{T}\int_{0}^{T}\varepsilon _{\text{viscous}}+\varepsilon
    _{\text{turb}} \ dt\notag \\
    \leq \frac{1}{T}\left\langle \frac{1}{2}\frac{1}{|\Omega |}
    ||u^{0}||^{2}\right\rangle _{e}+\left\langle \left\langle \frac{1}{|\Omega |}(f,u(t))\,\right\rangle _{e}\right\rangle _{T}.%
\end{align}
The uniform bounds in Proposition (\ref{EQ:bounds})\ imply that the first three terms
satisfy 

\begin{align*}
    \frac{1}{T}\left\langle \frac{1}{2}\frac{1}{|\Omega |}||u(T)||^{2}\right
    \rangle _{e}
    & =\mathcal{O}\left( \frac{1}{T}\right) , \\
    \frac{1}{T}\left\langle \frac{1}{2}\frac{1}{|\Omega |}
    ||u^{0}||^{2}\right\rangle _{e}
    & =\mathcal{O}\left( \frac{1}{T}\right) , \\
    \frac{1}{T}\int_{0}^{T}\varepsilon _{\text{viscous}}+\varepsilon _{\text{turb}}dt
    & =\left\langle \varepsilon \right\rangle _{T}\leq C(data)<\infty.
\end{align*}%
Consider the second term, $\left\langle \left\langle \frac{1}{|\Omega |}%
(f,u)\,\right\rangle _{e}\right\rangle _{T}$, on the RHS of (\ref{EQ:first}%
). Since $f=f(x;\omega _{j})$ is independent of time, the Cauchy-Schwarz
inequality in $L^{2}(0,T)$ and $\left\langle \left\langle \cdot
\right\rangle _{e}\right\rangle _{T}=\left\langle \left\langle \cdot
\right\rangle _{T}\right\rangle _{e}$\ yields 
\begin{align*}
\left\langle \left\langle \frac{1}{|\Omega |}(f,u)\,\right\rangle
_{e}\right\rangle _{T}\leq \left\langle \left\langle \frac{1}{|\Omega |}%
||f||^{2}\,\right\rangle _{e}\right\rangle _{T}^{1/2}\left\langle
\left\langle \frac{1}{|\Omega |}||u||^{2}\right\rangle _{e}\right\rangle
_{T}^{1/2} 
\leq F U_T.
\end{align*}%

The above estimates on the terms in (\ref{EQ:first}) thus imply
    \begin{equation}
\left\langle \varepsilon \right\rangle _{T}
\leq \mathcal{O}\left(\frac{1}{T}%
\right)+F U_T.  \label{eq:NewStep1}
\end{equation}
To bound $F$ in terms of flow quantities, 
take the inner product of (\ref%
{eq:EEVmodel}) with $f(x;\omega _{j})$, integrate the nonlinear term by
parts, ensemble average and time average over $[0,T]$. This yields%
\begin{align}
\label{eq:Step2}
    &F^{2}
    =\frac{1}{T}\left\langle \frac{1}{|\Omega |}(u(T)-u^{0},f)\right\rangle_{e}
    -\left\langle \left\langle \frac{1}{|\Omega |}(uu,\nabla f)\right\rangle_{e}\right\rangle _{T} \notag\\
    &\hspace{1cm}+\left\langle \left\langle \frac{1}{|\Omega |}\int_{\Omega }\nu \nabla
    u:\nabla f\,dx\right\rangle _{e}\right\rangle _{T} 
    +\left\langle \left\langle 
    \frac{1}{|\Omega |}\int_{\Omega }\mu \tau |u^{\prime }(x,t)|_{e}^{2} \, \nabla
    u:\nabla f\,dx\right\rangle _{e}\right\rangle _{T}.
\end{align}
The first term on the RHS is\ $\mathcal{O}(1/T)$ by (\ref{EQ:bounds}). The
second and third on the RHS are bounded by the Cauchy-Schwarz inequality and
(\ref{eq:FandLproperties}). Thus, for any $0<\beta <1$ we have%
\begin{align}
\label{EQseconds}
    \textbf{Second term:} \quad 
    \left\vert \left\langle \left\langle \frac{1}{|\Omega |}(uu,\nabla
    f)\right\rangle _{e}\right\rangle _{T}\right\vert 
    &\leq \left\langle
    \left\langle ||\nabla f(\cdot ;\omega _{j})||_{\infty }\frac{1}{|\Omega |}%
    ||u||^{2}\right\rangle _{e}\right\rangle _{T} \notag \\ 
    &\leq \left( \max_{j}||\nabla f(\cdot ;\omega _{j})||_{\infty }\right)
    \left\langle \left\langle \frac{1}{|\Omega |}||u||^{2}\right\rangle
    _{e}\right\rangle _{T}\notag \\ 
    &\leq\frac{F}{L}  U_T^2 
\end{align}
\begin{align}
\label{EQthirdz}
    \textbf{Third term:} \quad 
     &\left\langle \left\langle \frac{1}{|\Omega |}\int_{\Omega }\nu
    \nabla u(x,t;\omega _{j}):\nabla f(x;\omega _{j}) \,dx\right\rangle
    _{e}\right\rangle _{T}   \notag\\
    &\leq \left\langle \left\langle \frac{\nu ^{2}}{|\Omega |}||\nabla
    u||^{2}\right\rangle _{e}\right\rangle _{T}^{\frac{1}{2}}\left\langle
    \left\langle \frac{1}{|\Omega |}||\nabla f||^{2}\right\rangle
    _{e}\right\rangle _{T}^{\frac{1}{2}} \notag \\ 
    &\leq \left\langle \varepsilon _{\text{viscous}}\right\rangle _{T}^{\frac{1}{2}}%
    \sqrt{\nu }\frac{F}{L} \notag \\ 
    &  
    \leq \frac{\beta}{2} \frac{F}{U_T}  \left\langle \varepsilon_{\mathrm{\text{viscous}}} \right\rangle_{T}
    + \frac{1}{2\beta} U_T F \frac{\nu}{L^{2}}.
\end{align}
The fourth term on the RHS is estimated by successive applications of the
space, time and ensemble Cauchy-Schwarz inequality as follows
\begin{align}
    & \hspace{1.5cm}\textbf{Fourth term:} \quad 
     \left\vert \left\langle \left\langle \frac{1}{|\Omega |}%
    \int_{\Omega }\mu \tau |u^{\prime }|_{e}^{2}\nabla u:\nabla f \,dx\right\rangle
    _{e}\right\rangle _{T}\right\vert   \notag \\
    & \hspace{4cm} \leq \left\langle \left\langle \frac{1}{|\Omega |} \int_{\Omega } \left( \sqrt{%
    \mu \tau }|u^{\prime }|_{e}\right) \left( \sqrt{\mu \tau }|u^{\prime
    }|_{e}|\nabla u|\right) |\nabla f|\,dx\right\rangle _{e}\right\rangle _{T} \notag\\
    &\leq \max_{j}||\nabla f||_{L^{\infty }}\left\langle \left\langle \left( 
    \frac{1}{|\Omega |}\int_{\Omega }\mu \tau |u^{\prime }|_{e}^{2} \,dx\right)
    ^{1/2}\cdot \left( \frac{1}{|\Omega |}\int_{\Omega }\mu \tau |u^{\prime
    }|_{e}^{2}|\nabla u|^{2}\,dx\right) ^{1/2}\right\rangle _{e}\right\rangle
    _{T}.
\end{align}

Since $\max_{j}||\nabla f||_{L^{\infty }}\leq F/L$ (yet more) application of
standard inequalities yields

\begin{align}
    &\textbf{Fourth term:} \quad
    \left\vert \left\langle \left\langle \frac{1}{|\Omega |}\int_{\Omega }\mu
    \tau |u^{\prime }|_{e}^{2}\nabla u:\nabla f \,dx\right\rangle _{e}\right\rangle
    _{T}\right\vert \notag \\
    &\leq \sqrt{\mu \tau }\frac{F}{L}\frac{1}{T}\int_{0}^{T}\left\langle \frac{1}{%
    |\Omega |}\int_{\Omega }|u^{\prime }|_{e}^{2}\,dx\right\rangle
    _{e}^{1/2}\left\langle \frac{1}{|\Omega |}\int_{\Omega }\mu \tau |u^{\prime
    }|_{e}^{2}|\nabla u|^{2}\,dx\right\rangle _{e}^{1/2}dt \notag\\ 
    &\leq \sqrt{\mu \tau }\frac{F}{L}\left\langle \left\langle \frac{1}{|\Omega |}%
    \int_{\Omega }|u^{\prime }|_{e}^{2}\,dx\right\rangle _{e}\right\rangle
    _{T}^{1/2}\left\langle \left\langle \frac{1}{|\Omega |}\int_{\Omega }\mu
    \tau |u^{\prime }|_{e}^{2}|\nabla u|^{2}\,dx\right\rangle _{e}\right\rangle
    _{T}^{1/2} \notag\\ 
    &\leq \sqrt{\mu \tau }\frac{F}{L}\left\langle \left\langle \frac{1}{|\Omega |}%
    ||u^{\prime }||^{2}\right\rangle _{e}\right\rangle _{T}^{1/2}\left( \frac{1}{%
    T}\int_{0}^{T}\varepsilon _{\text{turb}}dt\right) ^{1/2},\\
    &= \sqrt{\mu \tau }\frac{F}{L} 
    U_T
    \left( \frac{1}{%
    T}\int_{0}^{T}\varepsilon _{\text{turb}}dt\right) ^{1/2},
\end{align}
as $|u^{\prime }|_{e}$ is independent of $\omega _{j}$. By the weighted
arithmetic-geometric mean inequality we thus (finally) have

\begin{align}
\label{EQfourthz}
    \textbf{Fourth term:} \quad
    &\left\vert \left\langle \left\langle \frac{1}{|\Omega |}%
    \int_{\Omega }\mu \tau |u^{\prime }(x,t)|_{e}^{2}\nabla u(x,t;\omega
    _{j}):
    \nabla f(x;\omega _{j})\,dx\right\rangle _{e}\right\rangle
    _{T}\right\vert   \notag\\
    &\leq \frac{\beta }{2}\frac{F}{U_T}\left\langle \varepsilon
    _{\text{turb}}\right\rangle _{T}+\frac{\mu \tau }{2\beta }\frac{U_TF}{L^{2}}%
    \left\langle \left\langle \frac{1}{|\Omega |}||u^{\prime
    }||^{2}\right\rangle _{e}\right\rangle _{T}\notag.
\end{align}

Using the three estimates\ (\ref{EQseconds}), (\ref{EQthirdz}), (\ref%
{EQfourthz}) in the (\ref{eq:Step2}) for $F^{2}$ yields%
\begin{align*}
F^{2} 
\leq \mathcal{O}\left( \frac{1}{T}\right) 
&+ \frac{F}{L}  U_T^2 
+\frac{\beta}{2}  \frac{F}{U_T} \left\langle \varepsilon_{\mathrm{\text{viscous}}} \right\rangle_{T}
+ \frac{1}{2\beta} U_T F \frac{\nu}{L^{2}} \\
&+\frac{\beta }{2}\frac{F}{U_T}\left\langle \varepsilon_{\text{turb}}\right\rangle _{T}+\frac{\mu \tau }{2\beta }\frac{U_TF}{L^{2}} {U_T^{\prime}}^2.
\end{align*}%

Canceling one $F$, multiplying by $U_T$ and collecting terms gives an
estimate for \\
$FU_T$ on the RHS of\ (\ref%
{eq:NewStep1})%
\begin{align*}
    FU_T \leq U_T\, \mathcal{O}\left(\frac{1}{T}\right) + \frac{1}{L} U_T^3
    +\frac{\beta}{2} \langle \varepsilon \rangle_T 
    + \frac{U_T}{2\beta} \frac{U_T\nu}{L^2} +\frac{1}{2\beta} \frac{U_T}{L^2} \mu \tau U_T'^2U_T
\end{align*}

\begin{align}
    \left\langle \varepsilon \right\rangle _{T}
    &\leq \mathcal{O}\left( \frac{1}{T}%
    \right) +U_T\, \mathcal{O}\left(\frac{1}{T}\right) + \frac{1}{L} U_T^3
    +\frac{\beta}{2} \langle \varepsilon \rangle_T 
    + \frac{U_T}{2\beta} \frac{U_T\nu}{L^2} +\frac{1}{2\beta} \frac{U_T}{L^2} \mu \tau U_T'^2U_T
\end{align}

The limit superior as $T\rightarrow \infty $ of the last inequality, which
exists by Proposition (\ref{EQ:bounds}), yields the following 

\begin{align*}
    \left\langle \varepsilon \right\rangle _{\infty }
    & =\limsup_{T\rightarrow
    \infty }\frac{1}{T}\int_{0}^{T}\varepsilon _{\text{viscous}}+\varepsilon
    _{\text{turb}} \, dt\\
    &\leq \frac{U^{3}}{L}
    +\frac{\beta }{2}\left\langle \varepsilon \right\rangle _{\infty }+\frac{1}{%
    2\beta }\frac{\nu U^{2}}{L^{2}}+\frac{1}{2\beta }\frac{U}{L^{2}}\mu \tau
    \left( U^{\prime }\right) ^{2}U.
\end{align*}%
Rearranging gives%
\begin{equation*}
\left( 1-\frac{\beta }{2}\right) \left\langle \varepsilon \right\rangle
_{\infty }\leq \frac{U^{3}}{L}\left( 1+\frac{1}{2\beta }\frac{\nu }{LU}+%
\frac{\mu }{2\beta }\frac{\tau U^{\prime }}{L}\frac{U^{\prime }}{U}\right) .
\end{equation*}
Using $T^{\ast }=L/U$, $\alpha =\beta /2,I(u)=(U^{\prime }/U)^{2}$ completes
the proof:%
\begin{equation*}
\left\langle \varepsilon \right\rangle _{\infty }\leq \frac{U^{3}}{L}\left( 
\frac{1}{1-\alpha }+\frac{1}{4\alpha \left( 1-\alpha \right) }\mathcal{R}%
e^{-1}+\frac{\mu }{4\alpha \left( 1-\alpha \right) }\frac{\tau }{T^{\ast }}%
I(u)\right) .
\end{equation*}

\section{Model Existence}
\label{sec_model_existence}
We briefly consider the question of existence of solutions for the model.
For EV models generally, existence is a challenging problem due to the
(generally) non-monotone nonlinearity in the highest derivative terms $\nu
_{\text{turb}}$\ introduces. Suppressing non-essential features, there is a
fundamental issue of the meaning of a term in a weak form like%
\begin{equation*}
\int_{0}^{T}\int_{\Omega }\nu _{\text{turb}}(u)\nabla u(x,t):\nabla \phi (x,t)\,dx\,dt.
\end{equation*}%
With $\nu _{\text{turb}}=\mu \tau |u^{\prime }|_{e}^{2}$\ we have $\nu _{\text{turb}}\in
L^{\infty }(0,T;L^{1}(\Omega ))$. As $\nu _{\text{turb}}\in L^{\infty
}(0,T;L^{1}(\Omega ))$ and $\nabla u\in L^{2}(0,T;L^{2}(\Omega ))$,\ this
term is not well defined even for $\phi \in C^{\infty }(\Omega \times (0,T))$
. There are (at least) two natural responses to the problems this creates
for an existence theory. The first is to develop further the theory to
broader notions of solution or more specialized \'{a} priori estimates.
There has been slow but steady progress in this important direction,
summarized (up to 2014) in Rebollo and Lewandowski \cite{RL14}. The second,
taken in this section, is to interrogate the model and see if the
mathematical difficulties are resolved in a model of greater physical
fidelity.

The \'{a} priori bounds in Proposition \ref{EQ:bounds} imply $|u^{\prime }|_{e}(x,t)\in
L^{2}(\Omega )$. However, this does not imply that the turbulence length
scale $l(x,t)=|u^{\prime }|_{e}\tau $\ is bounded. An unbounded $l(x,t)$ can
lead to the absurd situation where fluctuating structures are (in some
regions) farther apart than the size of the domain. It thus is physically
sensible to modify $l(x,t)$ by capping it at the domain size (or some other
intermediate length), by a hard cap (or a smooth transition):%
\begin{equation}
\begin{array}{cc}
\text{hard cap:} & l(x,t)=\min \{|u^{\prime }|_{e}\tau ,L_{\Omega }\},%
\end{array}
\label{EQ:cappping_l(x,t)}
\end{equation}%
Existence of distributional solutions now follows from Theorem 3.1 p.8 in 
\cite{LL02}.

\begin{theorem}
For the model (\ref{eq:EEVmodel}) suppose $l(x,t)$\ is replaced by (\ref%
{EQ:cappping_l(x,t)}). Then, there exists at least one distributional
solution.
\end{theorem}

\begin{proof}
We now have $\nu _{\text{turb}}(x,t)=\mu |u^{\prime }|_{e}\min \{|u^{\prime
}|_{e}\tau ,L_{\Omega }\}$. The proof is an application of Theorem 3.1 p.8
in \cite{LL02} for which three conditions (called H1, H2, H3 therein) must
be verified for $A(u)=\nu +\mu |u^{\prime }|_{e}\tau \min \{|u^{\prime
}|_{e}\tau ,L_{\Omega }\}.$ We will verify a notationally simpler version of
these that implies H1,H2 and H3 for an $A(u)$ where we suppress
non-essential features. Specifically, let 
\begin{equation*}
A(u):=\nu +|u|\min \{|u|,1\}.
\end{equation*}%
Hypothesis H1 is implied by: for $u\in L^{\infty }(0,T;L^{2}(\Omega ))\cap
L^{2}(0,T;H^{1}(\Omega ))$, it follows that $A(u)\in L^{\infty
}(0,T;L^{2}(\Omega ))$ and 
\begin{equation*}
||A(u)||_{L^{\infty }(0,T;L^{2}(\Omega ))}\leq C\left( 1+||u||_{L^{\infty
}(0,T;L^{2}(\Omega ))}\right) .
\end{equation*}%
This is clearly true since a.e. $|u|\min \{|u|,1\}\leq 1\cdot |u|$.

Hypothesis H2 is implied by: for some $a_{0}>0,A(u)\geq a_{0}$. This holds
with $a_{0}=\nu $.

Hypothesis H3 is implied by the following. For a sequence $u_{n}\in
L^{\infty }(0,T;L^{2}(\Omega ))\cap L^{2}(0,T;H^{1}(\Omega ))$ that
converges to $u_{\infty }\in L^{\infty }(0,T;L^{2}(\Omega ))\cap
L^{2}(0,T;H^{1}(\Omega ))$ weakly in $L^{2}(0,T;H^{1}(\Omega ))$ and
strongly in $L^{2}(\Omega \times (0,T))$ it follows that $A(u_{n})$
converges to $A(u_{\infty })$ strongly in $L^{2}(\Omega \times (0,T))$.
Condition H3 follows (after obvious rescaling) from the next lemma.
\end{proof}

\begin{lemma}
Let $A(u):=\nu +|u|\min\{|u|,1\}.$\ For a sequence $u_{n}\in L^{\infty
}(0,T;L^{2}(\Omega ))$ that converges to $u_{\infty }\in L^{\infty
}(0,T;L^{2}(\Omega ))$ strongly in $L^{2}(\Omega \times (0,T))$, it follows
that 
\begin{equation*}
A(u_{n})\rightarrow A(u_{\infty })\quad \text{strongly in }L^{2}(\Omega
\times (0,T)).
\end{equation*}
\end{lemma}

\begin{proof}
Let 
\begin{equation*}
m_{n}:=\min\{|u_{n}|,1\} \quad \text{and} \quad m_{\infty }:=\min\{|u_{\infty }|,1\}.
\end{equation*}%
Since the function $x\rightarrow \min\{|x|,1\}$ is Lipschitz with constant 1, 
$u_{n}\rightarrow u_{\infty }$ strongly in $L^{2}(\Omega \times (0,T))$
implies $m_{n}\rightarrow m_{\infty }$ strongly in $L^{2}(\Omega \times
(0,T))$ as%
\begin{equation*}
||m_{n}-m_{\infty }||_{L^{2}(\Omega \times (0,T))}\leq ||u_{n}-u_{\infty
}||_{L^{2}(\Omega \times (0,T))}.
\end{equation*}%
Let $ N\in\mathbb{N}$ such that $\forall n>N,$ 
\begin{align*}
    ||{u_n -u_\infty}||_{L^2(\Omega \times (0,T))}^2 
    \leq \frac{\varepsilon^2}{12}.
\end{align*}
Partition $\Omega \times (0,T)$ into three disjoint sets, for any $n>N$, as follows: 
\begin{align*}
    A_n = \{(x, &t ) : |u_n| \leq 1\}, 
    B = \{ (x, t): |u_n|>1, |u_\infty|\geq1\} \\
    &\text{and } C = \{ (x, t): |u_n|>1, |u_\infty|<1, \}
\end{align*} %
We need to estimate the following integral:
\begin{align*}
    \int_{\Omega \times (0,T)} |A(u_{n})-A(u_{\infty })|^{2}\,dx\,dt
    =\left( \int_{A_n}+\int_{B}+\int_{C}\right) ||u_n|m_n-|u_\infty|m_\infty|^{2}\,dx\,dt \\
\end{align*}%
For the integrals over $A_n$ and $B$, subtract and add $|u_{n}|m_{\infty }$. Since $|u_{n}|\leq 1$ and $0\leq m_{\infty }\leq1$ on $A_n$, 
\begin{gather*}
2\int_{A_n}|u_{n}|^{2}|m_{n}-m_{\infty }|^{2}\,dx\,dt+2\int_{A_n}||u_{n}|-|u_{\infty
}||^{2}|m_{\infty }|^{2}\,dx\,dt \\
\leq 2\int_{A_n}|m_{n}-m_{\infty }|^{2}\,dx\,dt+2\int_{A_n}|u_{n}-u_{\infty
}|^{2}\,dx\,dt,
\end{gather*}%
which can be made smaller than, $\varepsilon ^{2}/3$. For the second integral on $B$, $m_{n}=m_{\infty }=1$, thus $%
m_{n}-m_{\infty }=0$ and%
\begin{gather*}
2\int_{B}|u_{n}|^{2}|m_{n}-m_{\infty }|^{2}\,dx\,dt+2\int_{B}||u_{n}|-|u_{\infty
}||^{2}|m_{\infty }|^{2}\,dx\,dt \\
\leq 2\int_{B}|u_{n}-u_{\infty
}|^{2}\,dx\,dt.
\end{gather*}%
This is also smaller than $\varepsilon ^{2}/3$. For
the third integral, subtract and add $m_n |u_\infty|$ and use the fact that $m_{n}=1$ and $m_{\infty }=|u_{\infty }|<1$ on $C$. Thus,%
\begin{align*}
    \int_{C} \left|
    |u_n| m_n -  |u_\infty |m_\infty \right|^2
    & \leq 2\int_{C}
    |u_n- u_\infty|^2 {|m_n|^2}
    +  2\int_{C} {|u_\infty|^2} |m_n- m_\infty |^2 
\end{align*}
which is again less than $\varepsilon^2/3.$
Thus, the sum of the three will be less than $\varepsilon^2 $ and convergence follows.
\end{proof}
We also note that the energy dissipation rate estimates of Theorem \ref{main_theorem} follow,
by the same proof, for the capped length scale.

\begin{corollary}
Suppose all the hypotheses of Theorem \ref{main_theorem} hold with $l(x,t)=|u^{\prime
}|_{e}\tau $\ replaced by $l(x,t)=\min \{|u^{\prime }|_{e}\tau ,L_{\Omega
}\} $. Then the conclusions of Theorem \ref{main_theorem} remain true.
\end{corollary}

\section{Conclusions}

The Constantin-Doering-Foias theory has yielded powerful results connecting
the theory of the Navier-Stokes equations to the physical phenomenology of
turbulence. Yet, its potential for providing analysis in support of
practical computation for turbulence models is even greater. Extending their
theory here has shown that EEV models faithfully replicate energy
dissipation rates when near wall effects are negligible. The question then
arises of the effect of walls and especially near-wall boundary layers where 
$\nabla u$\ is large. Looking carefully at the proof of Theorem \ref{main_theorem}, it is
clear that the same result holds (with essentially the same proof) under
no-slip boundary conditions provided the additional assumption $f(x)=0$ $on$ 
$\partial \Omega $ is made.

\begin{corollary}
Let $l(x,t)=\min \{|u^{\prime }|_{e}\tau ,L_{\Omega }\}$\ or $%
l(x,t)=|u^{\prime }|_{e}\tau $. Consider weak solutions satisfying the
energy inequality for the EEV model with periodic boundary conditions
replaced by no-slip boundary conditions on $\partial \Omega .$ Assume
additionally $f(x)=0$ $on$ $\partial \Omega ,$\ e.g. $f\in H_{0}^{1}\left(
\Omega \right) ^{3}$. Then the time and ensemble averaged rate of energy
dissipation for the ensemble eddy viscosity model satisfies the following.
For model parameter $\mu (\simeq 0.55)$ and selected model time scale $\tau
(<T^{\ast })$ and for any $0<\alpha <1,$%
\begin{equation*}
\left\langle \varepsilon \right\rangle _{\infty }\leq \left( \frac{1}{%
1-\alpha }+\frac{1}{4\alpha (1-\alpha )}\mathcal{R}e^{-1}+\frac{\mu }{%
4\alpha (1-\alpha )}\frac{\tau }{T^{\ast }}I(u)\right) \frac{U^{3}}{L}.
\end{equation*}
\end{corollary}

Heuristically, the above condition $f(x)=0$ on $\partial \Omega $ means that
boundary layers are weak enough that the periodic picture does not change.
The case of stronger boundary layers was studied using asymptotic analysis
by Speziale, Abid and Anderson \cite{SAA92}.\textsc{\ }It can also be
studied through energy dissipation rate estimates in shear flow following
the Constantin-Doering-Foias theory.\ Thus, the shear flow EEV case is an
important open problem.

The secondary failure model of eddy viscosity models is that they
cannot
account for intermittent energy flow from fluctuations back to means. The
evolution of this energy flow was studied in \cite{JL16} and coherent
extensions of EV models developed including this effect. Analysis of energy
dissipation rates in these expanded models is an important open problem.
Current computational resources limit the number $J$ of realizations
achievable (and Carati, Rogers and Wray \cite{CRW02} report good results
with $J=16$). As resources continue to expand, $J$ will increase so
understanding limits as $J\rightarrow \infty $\ is also an important open
problem.

\section*{Acknowledgments}

This work of William Layton and Nanda Nechingal Raghunathan was partially supported by the NSF grant DMS 2410893.

\end{document}